\documentclass[conference]{IEEEtran}

\usepackage{amsmath}    
\usepackage{graphics,epsfig,subfigure}    
\usepackage{verbatim}   
\usepackage{color}      
\usepackage{subfigure}  
\usepackage{amssymb}
\usepackage{wasysym}
\usepackage{epstopdf}
\usepackage{graphicx}
\usepackage{mathrsfs}

\newcommand{\prob}[1]{\text{Pr}\{#1\}}
\newcommand{\expect}[1]{\mathbb{E}\big\{#1\big\}}

\newcommand{\bv}[1]{{\boldsymbol{#1} }}
\newcommand{\script}[1]{{{\cal{#1} }}}

\newtheorem{lemma}{\textbf{Lemma}}

\newtheorem{proposition}{Proposition}

\allowdisplaybreaks
\begin{document}

\title{Optimizing Age-of-Information in a Multi-class Queueing System}
\author{\large{Longbo Huang}$^\dagger$ and \large{Eytan Modiano}$^*$\\
$^\dagger$longbohuang@tsinghua.edu.cn, IIIS@Tsinghua University\\
$^*$modiano@mit.edu,  MIT
}
 
\maketitle

\begin{abstract}
We consider the age-of-information in a multi-class $M/G/1$ queueing system, where each class generates packets containing status information.  Age of information is a relatively new metric that measures the amount of time that elapsed between status updates, thus accounting for both the queueing delay and the delay between packet generation.  This gives rise to a tradeoff between frequency of status updates, and queueing delay.  In this paper, we study this tradeoff in a system with heterogenous users modeled as a multi-class $M/G/1$  queue.  To this end, we derive the exact peak age-of-Information (PAoI) profile of the system, which measures the ``freshness'' of the status information. We then seek to optimize   the age of information, by formulating the problem using quasiconvex optimization, and obtain structural properties of the optimal solution. 
\end{abstract}


\section{Introduction} 
Realtime status information is critical for optimal control in many networked systems, such as sensor networks used to monitor temperature or other physical phenomenon    \cite{msra-sensor-2009}; autonomous vehicle systems, where accurate position information is required  to avoid collisions  \cite{self-driving-scheduling}; or wireless networks where realtime channel state information is needed to make  scheduling decisions \cite{delayed-scheduling}.  In all of these systems, what matters is not how fast the update information gets delivered, but rather, how accurately the received information describes the physical phenomenon being observed. 


Recently, \cite{yates-realtime-12}  proposes the notion of \emph{age-of-information} (AoI), which measures the average time between the generation of an update message until it is received by the control unit; thus measuring the ``freshness" of the available information.  The early works on age-of-information  consider homogeneous systems, where all entities use the same amount of resources for status update, and the length of status messages can be modeled using an i.i.d. exponential distribution. 

In this paper, we consider a heterogenous systems where  entities generate status messages with different length (service time) distributions.  In particular, we consider a multi-class M/G/1 queueing systems; where each entity generates status update messages according to a given distribution, and derive the exact \emph{peak age-of-information} (PAoI) value for each entity, which is the average maximum elapsed time since the latest received update packet is generated, and  captures the extent to which update information is delayed. 
We then consider a system where, for packet management, at most one packet can be kept in the system, and compute the  PAoI for the multi-class $M/G/1/1$ queue.   

Next, we turn our attention to the problem of optimizing the PAoI by controlling the arrival rate of update messages (i.e., the sampling rate of the physical process being observed).
We formulate the  optimization problem as a minimization of a quasiconvex cost function of the system age-of-information profile. We show that in the $M/G/1/1$ case, the optimization problem is a quasi-convex program and derive properties of the optimal solution. In the general $M/G/1$ case, however, the problem is a general non-convex program for which we derive an approximate solution. 
   
The notion of age-of-information was first introduced in  \cite{yates-realtime-12} in the context of a single source modeled as an M/M/1 queue, and extended to multiple sources in  \cite{yates-multiple-12}.  In \cite{yates-piggyback-vehicle-11, yates-aoi-vehicle-11} the authors consider the problem of minimizing the age of system state in vehicular networks. In \cite{ephremides-aoi-13} the AoI is analyzed for a system with random delays, and in \cite{tony-aoi-isit-14}  PAoI is derived for a single-class $M/M/1$ queueing system. 

Our work differs from these prior works in a number of ways.  First, we focus on the PAoI metric, which is closely related to the AoI,  but is much more tractable, thus facilitating its optimization.  Second, we consider general service time distributions, whereas previous works focus mainly on exponential service time. Finally, we minimize the {\it cost} of the PAoI in a system with heterogenous service requirements; where the service requirements of different entities are modeled using quasi-convex cost functions of the PAoI.

It is important to note here that AoI is different from the traditional delay metric considered in communication systems. Indeed, our results show that PAoI minimization is equivalent to minimizing the sum of update interval and update packet delay.  
Due to this difference, the ultimate optimization problem turns out to be non-convex.


This paper is organized as follows. In Section \ref{section:system} we present the system model. We derive the PAoI values for $M/G/1$ and $M/G/1/1$   in Section \ref{section:aoi}. We consider the system cost optimization problem in Section \ref{section:utility}, and present numerical results in Section \ref{section:num}.  We conclude the paper  in Section \ref{section:conclusion}. 


\section{System Model}\label{section:system}
We consider a system that consists of a set of $N$ entities, denoted by $\script{N}=\{1, ..., N\}$. To disseminate entity status information, the system regulates how frequently each entity updates its status. We denote this decision by an \emph{update rate} vector $\bv{\lambda}=(\lambda_1, ..., \lambda_N)$, where $\lambda_n$ is the update rate of entity $n$. We assume that the update process is Poisson  with rate $\lambda_n$. After generation, update packets are  relayed to a central unit for processing. To ensure that  each entity eventually  updates its status,  we require that $0<\lambda_{\min}\leq\lambda_n\leq\lambda_{\max}$ for all $n$. We denote by $\Lambda\triangleq\{\bv{\lambda}: \lambda_{\min}\leq\lambda_n\leq\lambda_{\max},\forall\, n\}$ the set of feasible rate vectors. 

%

\subsection{Single queue model}
In a practical system, different update information streams will share the limited system resources for delivery. We model this by the system shown in Fig. \ref{fig:queue}, where all update messages go through a single server queue.  This single server queueing model is the same as that adopted by prior work on AoI \cite{yates-realtime-12, yates-multiple-12, ephremides-aoi-13}.
\begin{figure}[cht]
\centering
\vspace{-.15in}
\includegraphics[height=0.6in, width=1.6in]{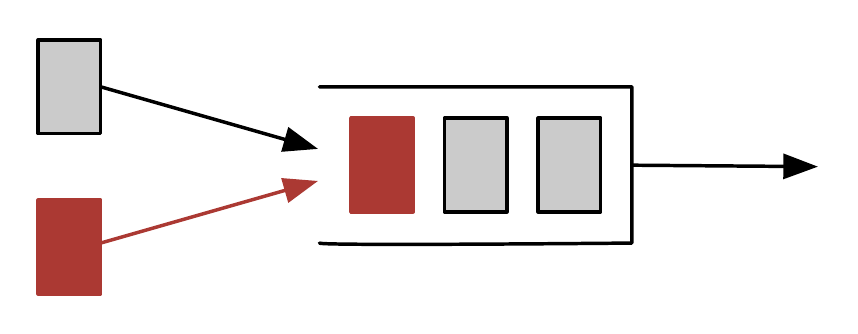}
\vspace{-.1in}
\caption{The update packet delivery process in a $2$-entity system. }\label{fig:queue}
\vspace{-.10in}
\end{figure}

In this queueing system, each new class-$n$ packet arrival to the queue represents the generation of a new entity $n$ update packet. 
Hence, class-$n$ packets arrive according to a Poisson process with rate $\lambda_n$. 
Departures from the queue, on the other hand, represent update packet reception events at their destinations. 
%
To model the heterogeneous resource requirements, e.g., entities may have update messages with different length distributions, we assume that each update packet from entity $n$ requires a random service time $X_n$, which is i.i.d.  with mean $x_n$ and  second moment $y_n$. 
For convenience, we denote $x_{\max}=\max_nx_n$ and $y_{\max}=\max_ny_n$.


This system is indeed equivalent to a multi-class $M/G/1$ queue. 
This model captures two key features of communication in networked systems: (i) resources are shared among different update streams, and (ii) queueing can occur during traffic delivery. Adopting this simple model allows us to focus on the age aspect of the update information.

\subsection{Age-of-Information}
The \emph{status age} of an entity at a particular time instance $t$ is defined to be the time elapsed since its latest received  update packet was generated. Fig. \ref{fig:peak-aoi-buffer} shows the status age (denoted by $\Delta_n(t)$) of entity $n$. The dropping points of $\Delta_n(t)$ are the time instances when an update packet is received, which resets the age value to a lower level (i.e., the current time minus the generation time of the new update packet).  
\begin{figure}[cht]
\centering
\vspace{-.15in}
\includegraphics[height=1.4in, width=3.0in]{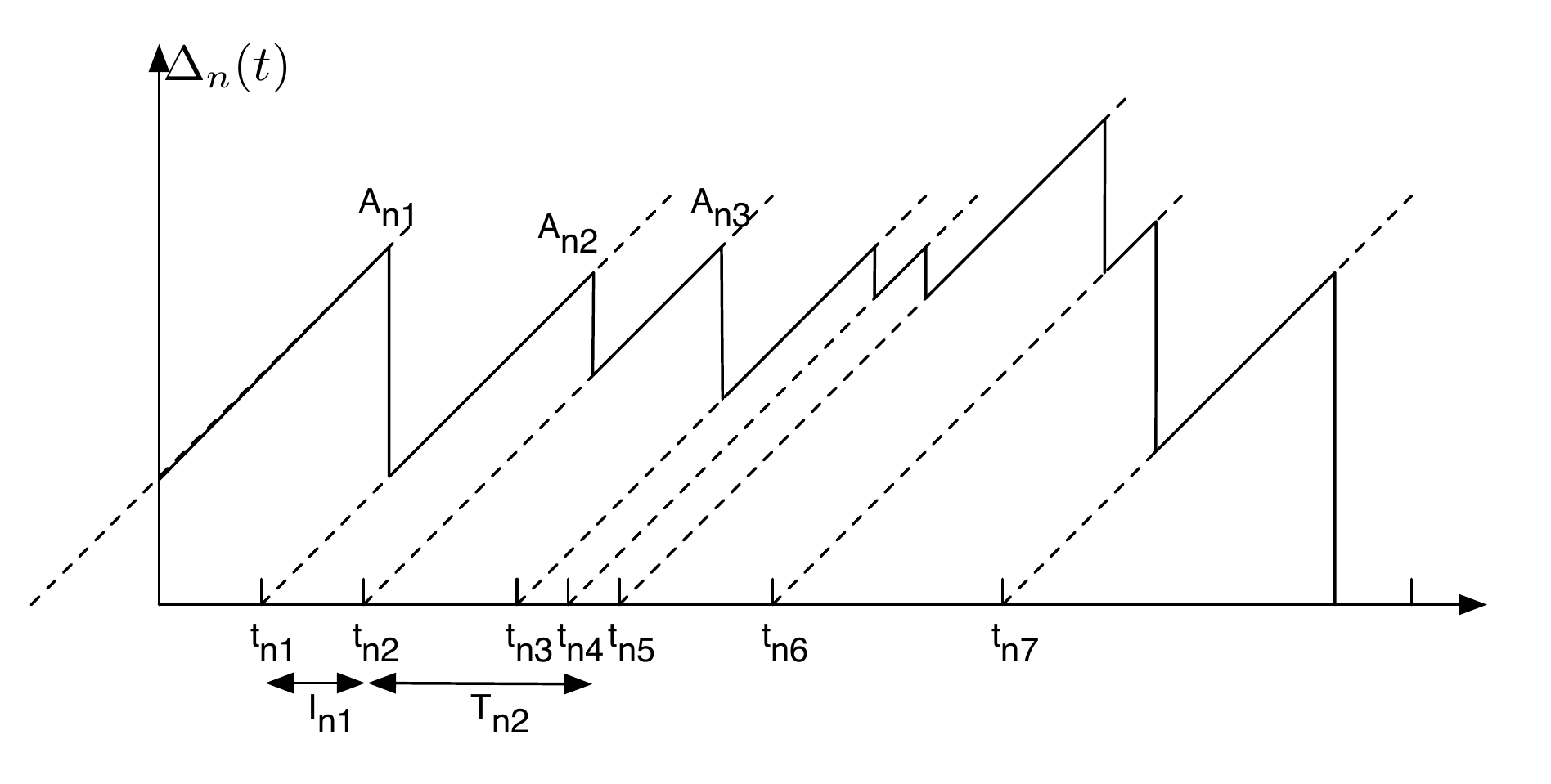}
\vspace{-.15in}
\caption{The evolution of the status age of  entity $n$ in the system. Here $A_{nk}$ denotes the $k$-th peak of age. }\label{fig:peak-aoi-buffer}
\vspace{-.15in}
\end{figure}

Given $\Delta_n(t)$, the average status age of entity $n$ is defined as: 
\begin{eqnarray}
A_{av, n} = \lim_{T\rightarrow\infty}\frac{1}{T}\int_{t=0}^T\Delta_n(t)dt. \label{eq:avg-aoi}
\end{eqnarray}
This metric is called the  \emph{age-of-information} (AoI) and was first considered in \cite{yates-realtime-12}. 
However, the AoI metric is hard to analyze. Moreover, in many systems, it is often the maximum status information delay that determines the   performance loss \cite{delayed-scheduling}. 
Thus,   we instead focus on the average \emph{peak} status age. 
Specifically, let $A_{nk}$ denote the $k$-th peak value of $\Delta_n(t)$ (See Fig. \ref{fig:peak-aoi-buffer}). The \emph{peak age-of-information} (PAoI) metric $A_n(\bv{\lambda})$ is defined as: 
\begin{eqnarray}
A_{p, n}(\bv{\lambda})\triangleq\lim_{K\rightarrow\infty}\frac{1}{K}\sum_{k=1}^{K}A_{nk}. 
\end{eqnarray}
Here we explicitly express PAoI as a function of the update rate vector $\bv{\lambda}$.  The PAoI metric was first considered in \cite{tony-aoi-isit-14} for the $M/M/1/1$ queue. It represents the maximum age of information before a new update is received. 
PAoI is closely related to the previously considered AoI metric $A_{av, n}$, e.g., in  \cite{yates-realtime-12} and \cite{yates-multiple-12}, but is much more tractable, thus facilitating its optimization.\footnote{Intuitively, PAoI provides an approximate upper bound for AoI, as it 
only considers the average sampled at the peak moments, whereas AoI in (\ref{eq:avg-aoi}) computes the time average  value of instantaneous age. }





\subsection{Optimizing update rates} 
We model the system performance  using a function of PAoI, to capture the fact that delay in status update often causes a proportional performance degradation.  
 Specifically, we consider the following 
 system cost function, i.e., 
\begin{eqnarray}
C_{\text{sys}}(\bv{A}(\bv{\lambda}))\triangleq\max_nC_n(A_{p, n}(\bv{\lambda})). 
\end{eqnarray}
Here  $C_n(A_n)$ is the cost of entity $n$ for having a PAoI of $A_n$, and it is assumed to be a quasiconvex and non-decreasing function in $A_n$ with $C_n(0)=0$. 
Our objective is to find an update rate vector $\bv{\lambda}\in\Lambda$ to minimize $C_{\text{sys}}(\bv{A}(\bv{\lambda}))$, i.e., minimize the maximum cost over all entities.  


\section{Computing  PAoI}\label{section:aoi}
In this section, we compute the PAoI for two different queueing models, i.e., the $M/G/1$ model and the $M/G/1/1$ model.  The two models differ in the way update packets are managed.  
In the $M/G/1$ model,  new packets are queued if the server is busy, while they are discarded in the $M/G/1/1$ model. This packet management scheme was also considered in \cite{tony-aoi-isit-14} for the $M/M/1$ model. 


\subsection{A general result for $G/G/1$ queues}
We first derive a useful result for general $G/G/1$ queues. 
 %
Denote by $I_n$ the inter-arrival time of entity $n$ packets, and let $W_{n}$ be the waiting time of  entity $n$ packets and let $T_n$ be the total sojourn time in the queue. We have the following proposition, in which  the superscript  ``$gg1$''  is used to indicate the relationship to  $G/G/1$ queues. 
\begin{proposition}
In a $G/G/1$ queue, the PAoI is given by: 
\begin{eqnarray}
A^{gg1}_{p, n} = \expect{I_n + T_n} = \expect{I_n + X_n+ W_n},\,\forall\,n.\quad\Diamond\label{eq:peak-aoi-q-gg1}
\end{eqnarray}
\end{proposition}
\begin{proof}
This relation follows from Fig. \ref{fig:peak-aoi-buffer}, where we see that the PAoI is equal to the time from the generation of an update packet until the completion of the  next update packet,  plus their inter-arrival time. 
\end{proof}

Equation (\ref{eq:peak-aoi-q-gg1}) shows that PAoI is indeed the sum of the update interval and update packet delay. 
For a multi-class $G/G/1$ queue, we also know that the AoI is given by \cite{yates-realtime-12}: 
\begin{eqnarray} 
A^{gg1}_{av, n} = \lambda_n\expect{I_nT_n + \frac{I_n^2}{2}}.\label{eq:avg-aoi-gg1}
\end{eqnarray}
Comparing (\ref{eq:avg-aoi-gg1}) and (\ref{eq:peak-aoi-q-gg1}), we have: 
\begin{eqnarray}
\hspace{-.4in} && A^{gg1}_{av, n}  - A^{gg1}_{p, n} \label{eq:general-peak-avg}\\
\hspace{-.4in}&&=  \lambda_n\bigg(  \expect{I_n (T_n + \frac{I_n}{2}) }  - \expect{I_n}\expect{I_n+T_n}\bigg)\nonumber\\
\hspace{-.4in}&&=\lambda_n\bigg(\expect{I_nT_n} - \expect{I_n}\expect{T_n} +\frac{\expect{I_n^2}}{2} -\expect{I_n}^2 \bigg). \nonumber
\end{eqnarray}
Equation (\ref{eq:general-peak-avg}) provides a way for checking how close the two metrics are to each other. 
It can be seen that when $I_n$ is a constant, i.e., periodic arrival,  
$A^{gg1}_{av, n}  = A^{gg1}_{p, n} - \frac{\lambda I^2_n}{2}$. We will  see in the next subsection that  $A^{gg1}_{av, n}  \leq A^{gg1}_{p, n} $ in the single class $M/M/1$ case.  Thus, PAoI serves as an upper bound for  AoI. 
More generally, the following lemma shows that  PAoI approximates AoI  for general single-class $G/G/1$ queues. 
\begin{lemma}\label{lemma:aoi-comparison}
For a general single-class $G/G/1$ queue,  
\begin{eqnarray*}
A^{gg1}_{p} - \frac{3\lambda\expect{I^2}}{2} - \lambda\expect{I}^2\leq A^{gg1}_{av}  \leq A^{gg1}_{p}+\lambda \expect{I^2}/2. \Diamond
\end{eqnarray*}
%
\end{lemma}
\begin{proof}
See Appendix A. 
\end{proof}

%

\subsection{PAoI for multi-class $M/G/1$ queue}
Using (\ref{eq:peak-aoi-q-gg1}), we can compute the PAoI for  each entity in the $M/G/1$ queue.

\begin{proposition}
The PAoI for a multi-class $M/G/1$ system is given by: 
\begin{eqnarray}
A^{mg1}_{p, n}
&=& \frac{1}{\lambda_n} + x_n + \frac{\sum_j\lambda_jy_j  }{2(  1-\sum_j\lambda_jx_j)}. \quad\Diamond\label{eq:peak-aoi-q}
\end{eqnarray}
\end{proposition}
In (\ref{eq:peak-aoi-q}) we have used the P-K formula to compute the waiting time in the  $M/G/1$ queue \cite{bertsekas_datanet}. It is necessary to ensure $\rho\triangleq\sum_j\lambda_jx_j<1$, so that the queue is stable and the PAoI is finite. 

We can use (\ref{eq:peak-aoi-q}) to compute the PAoI for a single class $M/M/1$ system.  In particular, with $N=1$ and exponential service time of rate $\mu$,  (\ref{eq:peak-aoi-q}) becomes: 
\begin{eqnarray}
A^{mm1}_p = 
 \frac{1}{\mu} \bigg( 1+\frac{1}{\rho}  + \frac{\rho}{1-\rho}  \bigg). \label{eq:peak-aoi-q-mm1}
\end{eqnarray}
In contrast,  the  AoI for $M/M/1$ derived in \cite{yates-realtime-12}, is given by,
\begin{eqnarray}
A_{av}^{mm1} = \frac{1}{\mu} \bigg( 1+\frac{1}{\rho}  + \frac{\rho^2}{1-\rho}  \bigg). \label{yates}
\end{eqnarray}
Comparing (\ref{eq:peak-aoi-q}) to (\ref{yates}), we have, 
\begin{eqnarray}
A^{mm1}_p - A^{mm1}_{av} =  \frac{1}{\mu}\rho = \frac{\lambda}{\mu^2}. 
\end{eqnarray} 
Hence, PAoI is a close upper bound of  AoI for $M/M/1$ queues, yet is much more tractable. 

{\bf Conservation Laws for PAoI:}
From (\ref{eq:peak-aoi-q}), we also obtain the following \emph{conservation} formula for PAoI: 
\begin{eqnarray}
\sum_n\lambda_nA^{mg1}_{p, n} &=&N + \rho+NW. \label{eq:peak-aoi-conservation}
\end{eqnarray}
We also see from (\ref{eq:peak-aoi-q}) that: 
\begin{eqnarray*}
A^{mg1}_{p, n}  - \frac{1}{\lambda_n} - x_n = A^{mg1}_{p, m}  - \frac{1}{\lambda_m} - x_m,\,\forall\,n, m,  
\end{eqnarray*}
which implies, 
\begin{eqnarray}
\frac{1}{\lambda_n} - \frac{1}{\lambda_m} = ( A^{mg1}_{p, n}  - x_n  ) - ( A^{mg1}_{p, m}   - x_m ). \label{eq:relation-aq}
\end{eqnarray}
Hence, the relationship between $A^{mg1}_{p, n}$ and $A^{mg1}_{p, m}$ is completely determined by $\lambda_n$ and $\lambda_m$. For example, $\lambda_m>\lambda_n$ implies $ A^{mg1}_{p, n}  - x_n > A^{mg1}_{p, m}   - x_m$.

\subsection{PAoI for $M/G/1/1$ queue}
Let us now consider the case when the server does not queue incoming update packets. This can be viewed as the server is performing packet management  \cite{tony-aoi-isit-14}. 
\begin{proposition}
The PAoI for a multi-class $M/G/1/1$ is given by: 
\begin{eqnarray}
A^{mg11}_{p, n} =x_n +  \frac{1}{\lambda_n} + \frac{\sum_{k}\lambda_kx_k}{\lambda_n},\,\forall\,n. \quad\Diamond\label{eq:a-final}
\end{eqnarray}
\end{proposition}
\begin{proof}
Let $Z_{jn}$ be the expected time to complete service for a class $n$ packet starting from the moment  a class $j$ packet begins receiving service at the server. 
We have: 
\begin{eqnarray}
A^{mg11}_{p, n}  = x_n+\frac{1}{\lambda}+\sum_j\frac{\lambda_j}{\lambda}Z_{jn}. \label{eq:a-def-gen}
\end{eqnarray}
Equation (\ref{eq:a-def-gen}) can be understood as follows. 
Since the status age of an entity decreases only when its next update packet is served,  the (expected)  peak age can be broken down into three components (see Fig. \ref{fig:peak-aoi-bufferless}). The first component $x_n$ in (\ref{eq:a-def-gen}) is the processing time of the current update packet. The second component is $\frac{1}{\lambda}$, the expected time needed  to get the next arrival (because packets arriving during busy periods are dropped). Then, the third component is the expected time needed until the completion of the next class $n$ update packet (the third term), which is $Z_{jn}$ if the next arrival turns out to be a class $j$ packet, resulting in an average time of $\sum_j\frac{\lambda_j}{\lambda}Z_{jn}$. 
\begin{figure}[cht]
\vspace{-.12in}
\centering
\vspace{-.12in}
\includegraphics[height=1.5in, width=2.7in]{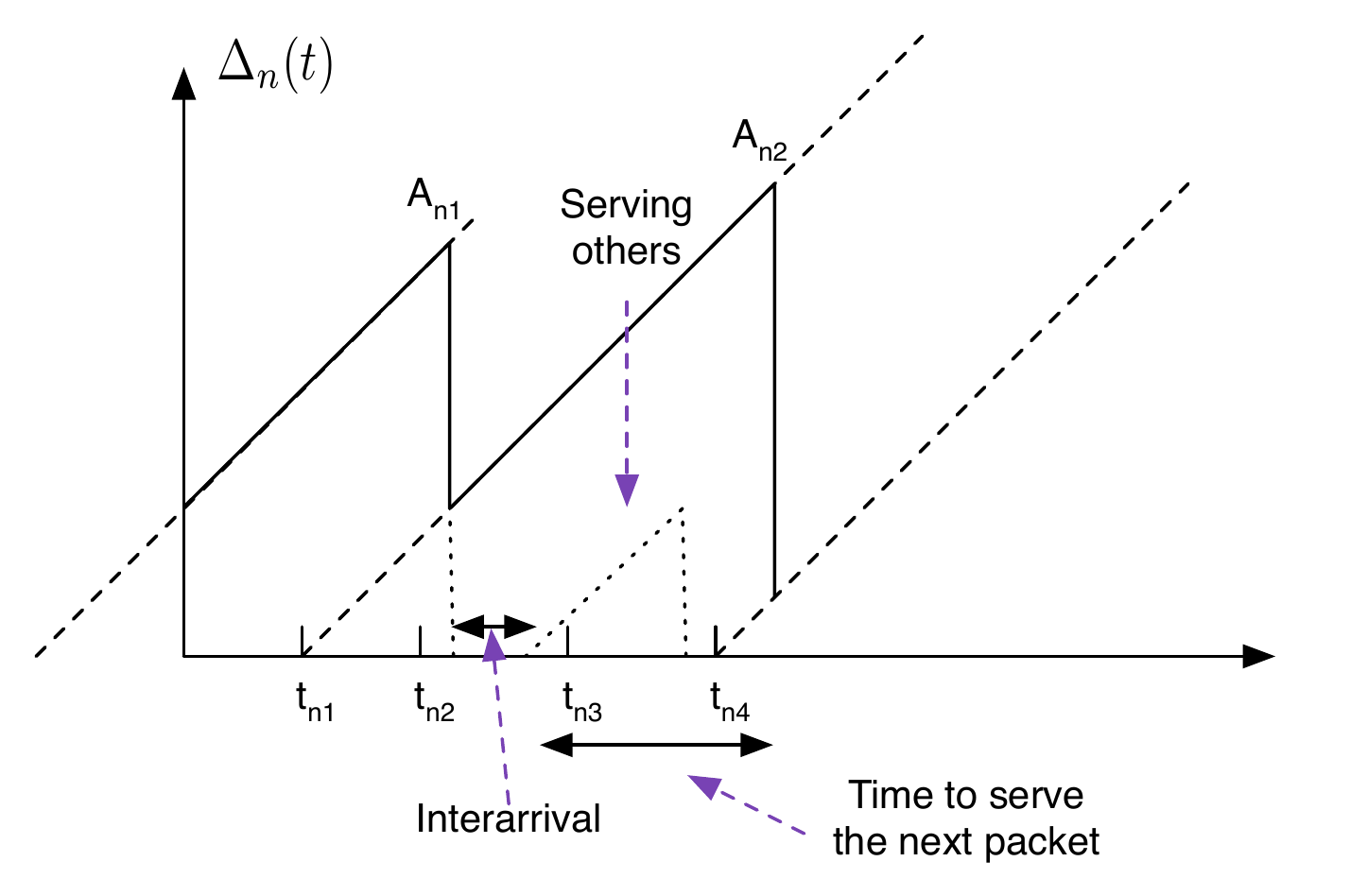}
\vspace{-.1in}
\caption{Evolution of entity $n$'s status age in the $M/G/1/1$ system. }\label{fig:peak-aoi-bufferless}
\vspace{-.1in}
\end{figure}

We now solve for $A^{mg11}_{p, n}$. Note that for any $i\neq n$, we have: 
\begin{eqnarray}
Z_{ji} = x_j +\frac{1}{\lambda} + \sum_k\frac{\lambda_k}{\lambda}Z_{ki}.\label{eq:z-def-gen}
\end{eqnarray}
Thus,  $Z_{ji}-Z_{ki}= x_j-x_k$ for any $j, k\neq i$. Plugging this into (\ref{eq:z-def-gen}) and using the fact that $Z_{ii}=x_i$, we obtain: 
\begin{eqnarray*}
Z_{ji} = x_j +\frac{1}{\lambda} + \frac{\lambda_i}{\lambda}x_i+\sum_{k\neq i}\frac{\lambda_k}{\lambda}[Z_{ji} +x_k-x_j].\label{eq:z-def-gen2}
\end{eqnarray*}
Therefore, 
\begin{eqnarray}
Z_{ji} = x_i + \frac{1}{\lambda_i} +x_j +\sum_{k\neq i}\frac{\lambda_kx_k}{\lambda_i} = x_j +\frac{1}{\lambda_i}  +\sum_{k}\frac{\lambda_kx_k}{\lambda_i}.  \nonumber 
\end{eqnarray}
Using this in (\ref{eq:a-def-gen}), we get: 
\begin{eqnarray}
\hspace{-.3in}&&A^{mg11}_{p, n} 
=x_n + \frac{1}{\lambda} + \frac{\lambda_n}{\lambda}x_n  +  \sum_{j\neq n}   \frac{\lambda_j}{\lambda} \bigg[ x_j + \frac{1}{\lambda_n} + \sum_{k}\frac{\lambda_kx_k}{\lambda_n} \bigg] \nonumber\\
\hspace{-.3in}&&\qquad\quad=2x_n + \frac{1}{\lambda_n}  + \frac{\sum_{k\neq i}\lambda_kx_k}{\lambda_n}. \nonumber
\end{eqnarray}
Rearranging the terms in the above gives (\ref{eq:a-final}). 
\end{proof}

It is interesting to note that  (\ref{eq:a-final}) does not require the second moment of service time, which is generally required in the analysis of $M/G/1$ queues.  This can be attributed to the fact that in the $M/G/1/1$ system, packets are never held in the buffer, thus the residual service time does  not play a role in the computation of  (\ref{eq:a-final}).
Also, when $N=1$,  (\ref{eq:a-final}) recovers the  result  from \cite{tony-aoi-isit-14} for the $M/M/1/1$ queue. 
It is also interesting to see in (\ref{eq:a-final}) that due to packet discard, the constraint $\rho<1$ can actually be violated. 

We  note from (\ref{eq:a-final}) that, for any achievable PAoI vector $\bv{A}^{mg11}_p=(A^{mg11}_{p1}, ..., A^{mg11}_{pN})$, 
\begin{eqnarray}
\lambda_nA^{mg11}_{p, n}- \lambda_nx_n  & = &  \sum_{k}\lambda_kx_k+1. 
\end{eqnarray}
Since the right-hand-side (RHS) does not depend on $n$, this implies that: 
\begin{eqnarray}
\lambda_nA^{mg11}_{p, n}- \lambda_nx_n & = & \lambda_mA^{mg11}_{p, m} - \lambda_mx_m. \label{eq:relation-mg11-case}
\end{eqnarray}
Similar to (\ref{eq:relation-aq}), the relationship of $A^{mg11}_{p, n}$ and $A^{mg11}_{p, m}$ is uniquely determined by $\lambda_n$ and $\lambda_m$. 
We also have the following conservation formula: 
\begin{eqnarray}
\sum_n\lambda_nA^{mg11}_{p, n}   & = & N + (N+1) \rho. \label{eq:peak-aoi-conservation-bufferless}
\end{eqnarray}
Comparing (\ref{eq:peak-aoi-q}) and (\ref{eq:a-final}), we have:  
\begin{eqnarray}
A^{mg11}_{p, n} -A^{mg1}_{p, n}  &=&  \frac{\sum_{k}\lambda_kx_k}{\lambda_n}- \frac{\sum_j\lambda_jy_j }{2(  1-\sum_j\lambda_jx_j)}. \label{eq:peak-diff}
\end{eqnarray}
This shows that $A^{mg11}_{p, n}$ can be much smaller than $A^{mg1}_{p, n}$ when the update rates are large, i.e., when $\rho$ is close to $1$. Thus, even though packets can be dropped in the $M/G/1/1$ system, such dropping may actually result in PAoI reduction as queueing delay is reduced.  
%


\vspace{-.1in}
\section{PAoI optimization}\label{section:utility}
Having computed the PAoI for the two cases,  we now consider the problem of  optimizing the update rates, i.e., minimize $C_{\text{sys}}(\bv{\lambda})$. This formulation enables us to provide differentiated service to different applications. 
In the following, we  start with the $M/G/1/1$ queue and then consider the $M/G/1$ queue. 

\subsection{$M/G/1/1$ optimization}
In this case, the utility optimization problem takes the following form: 
\begin{eqnarray}
\min_{\bv{\lambda}}: && C_{\text{sys}}(\bv{\lambda})=\max_nC_n(x_n + \frac{1+\sum_k\lambda_kx_k}{\lambda_n})\label{eq:obj-mg11}\\
\text{s.t.} 
&&\bv{\lambda}\in\Lambda. \nonumber
\end{eqnarray}
The following lemma shows that although (\ref{eq:obj-mg11}) is not convex, it can still be efficiently solved. 
\begin{lemma}
Problem (\ref{eq:obj-mg11}) is a quasiconvex program. $\Diamond$
\end{lemma}
\begin{proof}
First, we see that  $A_n= x_n + \frac{1+\sum_k\lambda_kx_k}{\lambda_n}$  is a linear-fractional function in $\bv{\lambda}$. Since each $C_n(A_n)$ function is quasiconvex and nondecreasing in $A_n$,   $C_n(A_n)$ is  quasiconvex in $\bv{\lambda}$. As the $\max$ operator preserves quasiconvexity, we conclude that $C_{\text{sys}}(\bv{\lambda})$ is also quasiconvex in $\bv{\lambda}$ and (\ref{eq:obj-mg11}) is a quasiconvex program over the convex set $\Lambda$ \cite{boydconvexopt}.  
\end{proof}

Therefore, the optimization problem (\ref{eq:obj-mg11}) can be solved by the bisection procedure described below \cite{boydconvexopt}.  
Define 
\begin{eqnarray}
\phi_t \triangleq \left\{\begin{array}{cc}
0 & C_{\text{sys}}(\bv{\lambda})\leq t\\
\infty & \text{else}. 
\end{array}\right. 
\end{eqnarray}
We see that $C_{\text{sys}}(\bv{\lambda})\leq t$ is equivalent to $\phi_t\leq0$, i.e.,  if   $\bv{\lambda}$ ensures $\phi_t\leq0$, it also ensures   $C_{\text{sys}}(\bv{\lambda})\leq t$.  
Hence, we then use the following bi-section algorithm to solve (\ref{eq:obj-mg11}).

\underline{$\mathtt{Bisection}$:} Set $l=0$ and $u=\max_nC_n(x_{\max} + \frac{1+N\lambda_{\max}x_{\max}}{\lambda_{\min}})$. Fix $\epsilon>0$. Then, repeat the following until $u-l\leq\epsilon$: 
\begin{enumerate}
\item Set $t=(l+u)/2$
\item Solve the following problem: 
\begin{eqnarray}
\min:\,\,1, \quad\text{s.t.} \,\, \,\phi_t\leq 0, \,\, \bv{\lambda}\in\Lambda. \label{eq:bisection}
\end{eqnarray}
\item If (\ref{eq:bisection}) is feasible, set $u=t$; otherwise set $l=t$. 
$\Diamond$
\end{enumerate}


%

%
\vspace{4pt}
Using (\ref{eq:relation-mg11-case}),  we also obtain the following properties of the optimal solution $\bv{\lambda}^*$ to  (\ref{eq:obj-mg11}), where $\succeq$  denotes entrywise larger. 
\begin{lemma}\label{lemma:mg11-pro}
Let $\bv{\lambda}^*$ be   an optimal solution of (\ref{eq:obj-mg11}). Then, (i) $\exists$ $\hat{\bv{\lambda}}\succeq\bv{\lambda}^*$ such that $\max_n\hat{\lambda}_n =\lambda_{\max}$ and $C_{\text{sys}}(\hat{\bv{\lambda}}) = C_{\text{sys}}(\bv{\lambda}^*)$,  and (ii) if all entities are identical,  i.e., $x_1=x_2$ and $C_n(A)=C_m(A)$, then $\lambda_n=\lambda_{\max}, \,\forall\,n$ is an optimal solution.  $\Diamond$
\vspace{4pt}
\end{lemma}
\begin{proof} First, note from (\ref{eq:relation-mg11-case}) that for a given  $\bv{\lambda}$, if we let $n_0=\arg\min_n (A_n-x_n)$, then we can express each $\lambda_m$ as: 
\begin{eqnarray}
\lambda_m  = \lambda_{n_0}\frac{A_{n_0}- x_{n_0}}{A_m - x_m},\,\,\forall\,m.  \label{eq:rate-A}
\end{eqnarray}
Consider an optimal solution $\bv{\lambda}^*$ and let $\bv{A}^*$ be the corresponding PAoI vector. We construct a $\hat{\bv{\lambda}}$ as follows. Denote $n^*_0=\arg\min_n (A^*_n-x_n)$. Then, we keep the ratio between any pair of rates fixed and proportionally increase all $\lambda_n$ until $\lambda_{n^*_0}=\lambda_{\max}$. 
From (\ref{eq:relation-mg11-case}) and (\ref{eq:a-final}), we see that $\hat{\bv{A}}\preceq\bv{A}^*$. Since each $C_n(A)$ is nondecreasing, we have $C_n(\hat{A}_n)\leq C_n(A_n^*)$, which implies $C_{\text{sys}}(\hat{\bv{\lambda}}) = C_{\text{sys}}(\bv{\lambda}^*)$. 
This proves (i). 

When all entities are identical, the rates must be the same for all entities. Hence, $\lambda_n=\lambda_{\max}$ for all $n$. 
\end{proof}

\subsection{$M/G/1$ optimization} 
In $M/G/1$, the optimization problem becomes: 
\begin{eqnarray}
\hspace{-.2in}\min: && C_{\text{sys}}(\bv{\lambda})=\max_nC_n(A^{mg1}_{p, n})\label{eq:obj-mg1}\\
\hspace{-.2in}\text{s.t.}  && \frac{1}{\lambda_n} + x_n + \frac{\sum_j\lambda_jy_j  }{2(  1-\sum_j\lambda_jx_j)}=A^{mg1}_{p, n},\,\forall\,n \label{eq:cond-mg1}\\
\hspace{-.2in}&&\sum_n\lambda_nx_n\leq1, \,\lambda_n>0. \nonumber 
\end{eqnarray}
Different from problem (\ref{eq:obj-mg11}), here the LHS in constraint (\ref{eq:cond-mg1}) is a sum of two linear-fractional functions, which may not be quasiconvex any more. 
Thus, to proceed, we approximate $A_n(\bv{\lambda})$ with another function $B_n(\bv{\lambda})$ defined as: 
\begin{eqnarray} 
B_n(\bv{\lambda})\triangleq 2\max\bigg(\frac{1}{\lambda_n} + x_n, \frac{\sum_j\lambda_jy_j }{2(  1-\sum_j\lambda_jx_j)}\bigg).  
\end{eqnarray} 
That is, we solve problem (\ref{eq:obj-mg1}) with $A^{mg1}_{p, n}$  replaced by $B_n$. The main advantage of introducing $B_n(\bv{\lambda})$ is that it is quasiconvex in $\bv{\lambda}$. Hence, the function  $C_{\text{sys}}(\bv{B}(\bv{\lambda}))\triangleq\max_nC_n(B_n)$ can  be efficiently minimized by the $\mathtt{Bisection}$ algorithm. 

We now look at the performance of the approximation. Define $\beta_n$ the maximum increasing slope of  $C_n(A)$, i.e., 
\begin{eqnarray*}
\beta_n\triangleq \inf\{\beta: |C_n(A_1) - C_n(A_2)|\leq \beta|A_1-A_2|, \,\forall\, A_1, A_2 \}. 
\end{eqnarray*}
We then have the following lemma, where $\bv{\lambda}^*_B$  is the optimal solution of the approximation program. 
%
\begin{lemma}\label{lemma:approx}
Let $\bv{\lambda}^*$ be an optimal solution of the original problem (\ref{eq:obj-mg1})  and denote $\bv{A}^*$ the resulting PAoI. Then, 
\begin{eqnarray}
C_{\text{sys}}(\bv{\lambda}^*)\leq C_{\text{sys}}(\bv{\lambda}^*_B) \leq C_{\text{sys}}(\bv{\lambda}^*) + \max_n\beta_nA^*_n. \quad \Diamond\label{eq:c-sys-approx}
\end{eqnarray}
\end{lemma}
\begin{proof}
See Appendix B. 
\end{proof}

When each $C_n(A)$ is linear in $A$, i.e., $C_n(A)=w_nA$, we have $\beta_n=w_n$. In this case, we can conclude from (\ref{eq:approx-ineq}) that $C_{\text{sys}}(\bv{\lambda}^*)\leq C_{\text{sys}}(\bv{\lambda}^*_B) \leq 2C_{\text{sys}}(\bv{\lambda}^*)$. 

%
%
%

\section{Numerical Results}\label{section:num}
We present a simple numerical example with $N=2$ entities,   and constant service times $x_1=1$ and $x_2=3$.  $\lambda_{\min}=0.01$ and $\lambda_{\max}=10$. The cost functions are given by $C_1(A_1)=4A_1^2$ and $C_2(A_2) = A_2^2$, and $C_{\text{sys}}(\bv{\lambda}) = \max(C_1, C_2)$. 

In Fig. \ref{fig:mg11-plot} we plot the cost values for the $M/G/1/1$ queue. The minimum value of $C_{\text{sys}}(\bv{\lambda})$ is achieved at $\lambda_1=10$ and $\lambda_2=6$, with a resulting $C_1=60.84$ and $C_{\text{sys}}(\bv{\lambda}) = C_2 = 61.36$. In this case, the PAoI vector is $\bv{A}=(3.9, 7.83)$ and the results match Lemma \ref{lemma:mg11-pro}. 
\begin{figure}[cht]
\centering
\includegraphics[height=1.3in, width=2.4in]{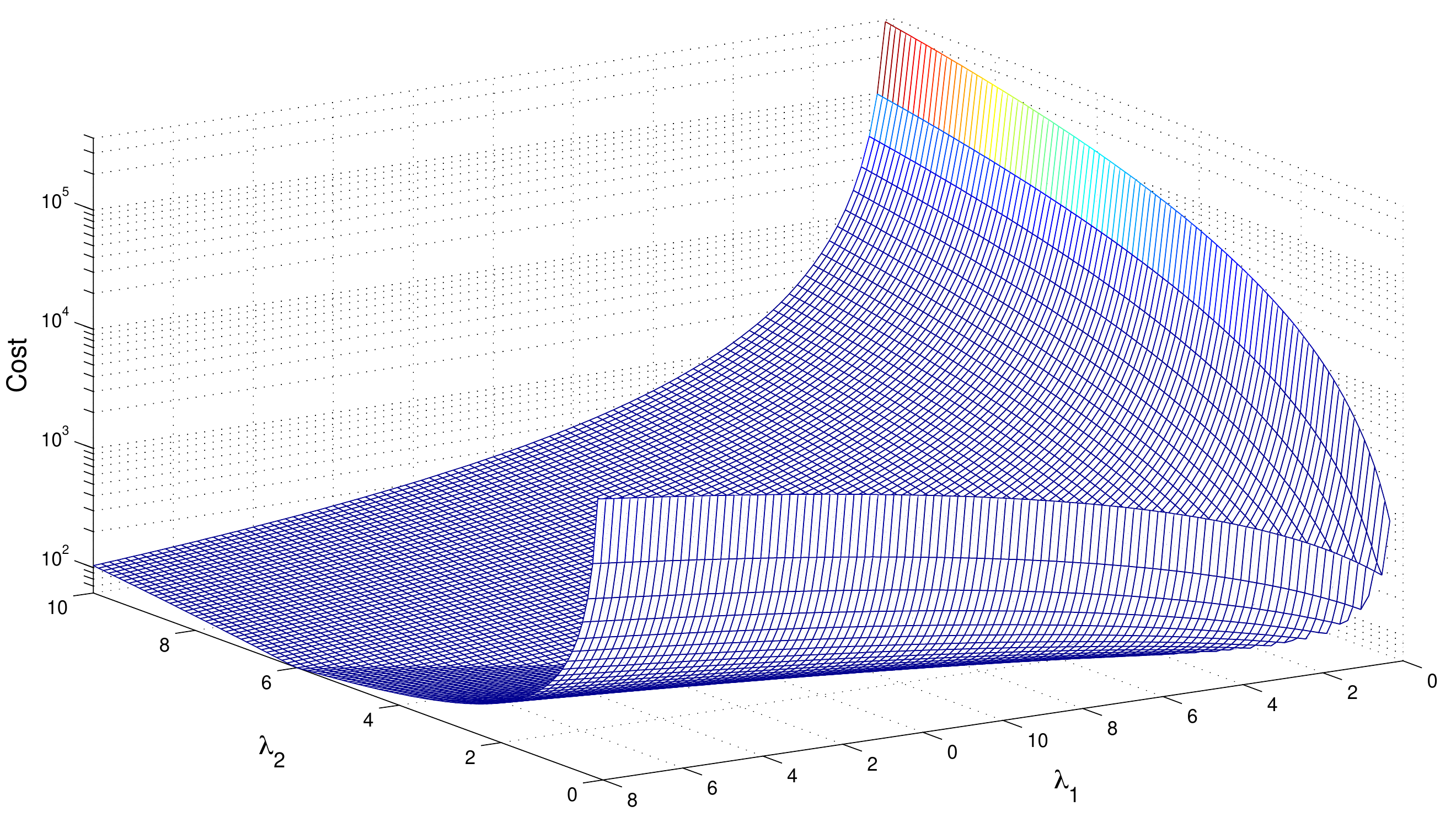}
\vspace{-.1in}
\caption{Cost value for the $M/G/1/1$ queue. }\label{fig:mg11-plot}
\vspace{-.12in}
\end{figure}

We next look at the $M/G/1$ case in Fig. \ref{fig:mg1-plot}. Since $\rho<1$ must be satisfied to ensure a finite PAoI, if a $\bv{\lambda}$ violates $\rho<1$, we  set its PAoI value to a constant (the flat region). In this case,  the minimum is achieved at $\bv{\lambda}=(0.29, 0.125)$ with $\bv{A}=(6.56, 13.11)$. Thus, $C_2=171.92$ and $C_1=C_{\text{sys}}(\bv{\lambda}) = 172.15$. It can be verified that (\ref{eq:relation-aq}) holds. 
%
\begin{figure}[cht] 
\centering
\vspace{-.15in}
\includegraphics[height=1.3in, width=2.4in]{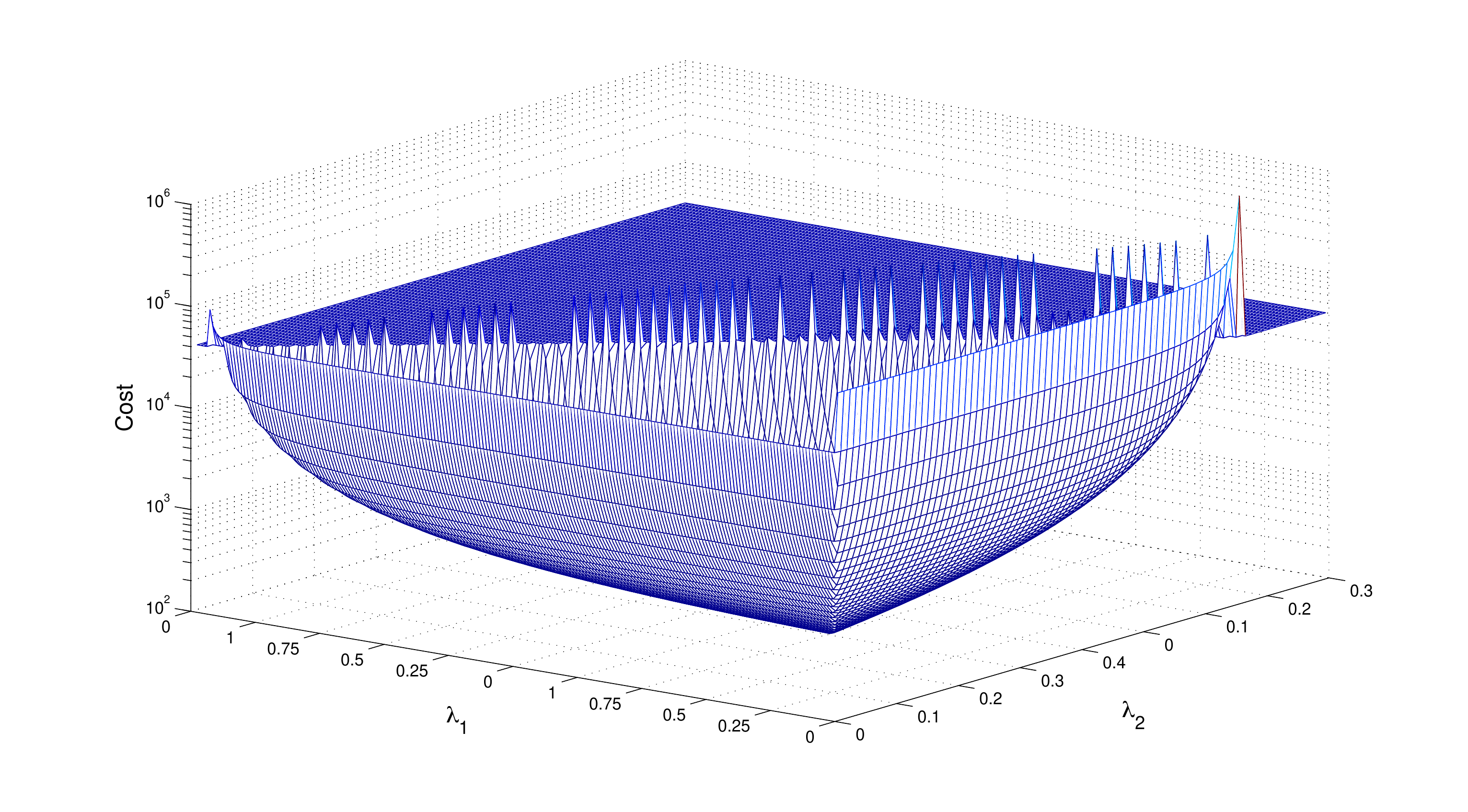}
\vspace{-.1in}
\caption{Cost value for the $M/G/1$ queue. }\label{fig:mg1-plot}
\vspace{-.12in}
\end{figure}

We also compute the optimal solution of the approximation approach to be $\bv{\lambda}^*_B = (0.285, 0.17)$. The  resulting PAoI vector is $\bv{A}=(8.94, 13.31)$ and the cost vector is $(C_1, C_2) = (319.69, 177.16)$, implying $C_{\text{sys}}(\bv{\lambda}^*_B) = 319.69$. It is not hard to verify that $C_{\text{sys}}(\bv{\lambda}^*_B)\leq 2C_{\text{sys}}(\bv{\lambda}^*)$. 

\section{Conclusion}\label{section:conclusion}
We study the age-of-information in a general multi-class $M/G/1$ queueing system.  The age-of-information is a new metric for system performance that represents not just the queueing delay, but also the delay in generating new information updates.  Our main contribution is to generalize the available results to systems with heterogeneous service time distributions, accounting for the fact that different entities may have different service requirements for their status updates.  We derive exact  peak-age-of-information expressions for both a $M/G/1$ system and $M/G/1/1$ system.  Using the PAoI measure allows us to optimize  system cost, as a function of PAoI,  by choice of the update interval.  
 
 \section{Acknowledgement}
This work was supported in part by the National Basic Research Program of China Grant 2011CBA00300, 2011CBA00301, the National Natural Science Foundation of China Grant 61033001, 61361136003, 61303195, Tsinghua Initiative Research Grant, and the China Youth 1000-talent Grant.  
The work of E. Modiano was supported by NSF Grant CNS-1217048 and ONR Grant N00014-12-1-0064.

\bibliographystyle{unsrt}
\bibliography{mybib}

 \section*{Appendix A -- Proof of Lemma \ref{lemma:aoi-comparison}}
 Here we prove Lemma \ref{lemma:aoi-comparison}. We drop all subscripts as $N=1$. 
 \begin{proof}
Since $T=W+X$ and $X$ is independent of $I$, we have $\expect{I_nT_n} - \expect{I_n}\expect{T_n} = \expect{IW} - \expect{I}\expect{W}$. 
Denote $T_Q$ the time it takes to clear the packets in the queue when a new packet arrives (not including the new arrival). Then, $W=(T_Q-I)^+$. Here $I$ is the inter-arrival time until the next packet arrives. 
Since $I$ is independent of $T_Q$, we have: 
\begin{eqnarray}
&& \expect{IW} \nonumber\\
&=& \int_{I}\int_{t=I}^{\infty} I (t-I)f(t)dtf(I)dI\nonumber\\
&=&\int_{I} \bigg(I \int_{t=I}^{\infty}tf(t)dt-I^2\prob{T_Q\geq I}  \bigg)f(I)dI\label{eq:iw-ineq0}\\
&\leq&\int_{I} I \int_{t=I}^{\infty}tf(t)dtf(I)dI.\label{eq:iw-ineq1} 
\end{eqnarray}
Using $(T_Q-I)^+ + I\geq T_Q$ and $I\geq0$, we have: 
\begin{eqnarray}
\int_{t=I}^{\infty}tf(t)dt\leq \expect{T_Q} \leq \expect{(T_Q-I)^+} +\expect{I}. 
\end{eqnarray}
Plugging this into (\ref{eq:iw-ineq1}), we get: 
\begin{eqnarray}
\expect{IW}  
&\leq&\int_{I} I \big(\expect{(T_Q-I)^+} +\expect{I}\big)f(I)dI  \nonumber \\
&=& \expect{I}\expect{(T_Q-I)^+} +\expect{I}^2.  
\end{eqnarray}
Using this in (\ref{eq:general-peak-avg}), we obtain: 
\begin{eqnarray}
A^{gg1}_{av} \leq  A^{gg1}_{p}+ \frac{\lambda \expect{I^2}}{2}. \nonumber
\end{eqnarray}

To derive the lower bound, we have from (\ref{eq:iw-ineq0}) that: 
\begin{eqnarray}
\hspace{-.4in}&& \expect{IW}  =\int_{I} \bigg(I \int_{t=I}^{\infty}tf(t)dt-I^2\prob{T_Q\geq I}  \bigg)f(I)dI \nonumber\\
\hspace{-.4in}&&\qquad\quad\,\,\,\geq \int_{I} I \bigg(\expect{T_Q}-\int_{t=0}^{I}tf(t)dt \bigg)f(I)dI -\expect{I^2}\nonumber\\
\hspace{-.4in}&&\qquad\quad\,\,\,\geq \int_{I} I \bigg(\expect{T_Q}- I  \bigg)f(I)dI -\expect{I^2}\nonumber\\
\hspace{-.4in}&&\qquad\quad\,\,\, \geq \expect{I} \expect{W} - 2\expect{I^2}. \label{eq:iw-ineq2}
\end{eqnarray}
In the last inequality, we have used $T_Q\geq W$. 
Plugging (\ref{eq:iw-ineq2}) into (\ref{eq:general-peak-avg}) proves the lower bound and completes the proof of the lemma. 
\end{proof}

\section*{Appendix B -- Proof of Lemma \ref{lemma:approx}}
We prove Lemma \ref{lemma:approx} here. 
\begin{proof}
From the definition of $B_n$, we see that given any $\bv{\lambda}$, 
$A^{mg1}_{p, n}(\bv{\lambda})\leq B_n(\bv{\lambda})\leq 2A^{mg1}_{p, n}(\bv{\lambda})$. Thus, we have for $\bv{\lambda}^*_B$ that: 
\begin{eqnarray}
\hspace{-.3in}&&C_{\text{sys}}(\bv{A}(\bv{\lambda}^*_B))\leq C_{\text{sys}}(\bv{B}(\bv{\lambda}^*_B))\label{eq:approx-ineq}\\
\hspace{-.3in}&&\qquad \qquad \quad\,\,\,\, \leq C_{\text{sys}}(\bv{B}(\bv{\lambda}^*))\leq C_{\text{sys}}(2\bv{A}(\bv{\lambda}^*)). \nonumber
\end{eqnarray}
Using the definition of $\beta_n$, we have for each $n$ that: 
\begin{eqnarray}
C_n(2A_n(\bv{\lambda}^*))\leq C_n(A_n(\bv{\lambda}^*) + \beta_nA^*_n. 
\end{eqnarray}
Taking the $\max$ over the above and combining it with (\ref{eq:approx-ineq}), we see that  (\ref{eq:c-sys-approx}) follows. 
\end{proof}

\end{document}